\documentclass[11pt]{article}
\usepackage{amsmath,amsthm,amssymb,amscd,fancybox,ifthen}
\usepackage[all]{xy}
\usepackage{times}

\newcommand\Def{\operatorname{Def}}

\newcommand\sign{\operatorname{sig}}

\newcommand\spnc{\operatorname{Spin}_{\bbC}}

\newcommand\fS{\mathfrak S}

\newcommand\cE{\mathcal E}

\newcommand\hH{\widehat{H}}

\newcommand\cC{\mathcal{C}}

\newcommand\cS{\mathcal{S}}

\newcommand\bZ{\overline{Z}}

\renewcommand\Im{\operatorname{Im}}

\newcommand\bbC{\mathbb C}

\newcommand\bbQ{\mathbb Q}
\newcommand\bbN{\mathbb N}

\newcommand\CI{{\mathcal C}^{\infty}}

\DeclareMathOperator{\Rind}{R-Ind}

\DeclareMathOperator{\Hom}{Hom}

\newtheorem{theorem}{Theorem}

\newtheorem{corollary}{Corollary}

\theoremstyle{definition}

\theoremstyle{remark}

\begin{document}

\title{Subelliptic $\spnc$ Dirac Operators, IV\\ Proof of the Relative Index Conjecture}

\author{Charles L. Epstein\footnote{ Research partially supported by NSF grants
    DMS02-03795, DMS06-03973, and the Thomas A. Scott chair.
    E-mail: cle@math.upenn.edu} \\ Department of Mathematics\\
  University of Pennsylvania}

\date{Date: February 29, 2012}

\maketitle
\begin{abstract} We prove the relative index conjecture, which in turn implies
  that the set of embeddable deformations of a strictly pseudoconvex
  CR-structure on a compact 3-manifold is closed in the $\CI$-topology.
\end{abstract}

\section{Proof of the Relative Index Conjecture}
In this short paper, which continues the analysis presented
in~\cite{Epstein44}, we show how the formula for the relative index between two
Szeg\H o projectors $\cS_0,\cS_1,$ defined by two embeddable CR-structures on a
contact 3-manifold $(Y,H),$ gives a proof of the relative index conjecture:
\begin{theorem}\label{thm1} Let $(Y,H)$ be a compact 3-dimensional co-oriented, contact
  manifold, and let $\cS_0$ be the Szeg\H o projector defined by an
  embeddable CR-structure with underlying plane field $H.$ There is an
  $M$ such that for the Szeg\H o projector $\cS_1$ defined by any
  embeddable deformation of the reference structure with the same
  underlying plane field, we have the upper bound:
    \begin{equation}\label{UprBnd}
    \Rind(\cS_0,\cS_1)\leq M.
  \end{equation}
\end{theorem}

Recall that the deformations of a reference CR-structure, $T^{0,1}_bY,$  on $(Y,H)$
are parameterized by
\begin{equation}
  \Def(Y,H,\cS_0)=\{\Phi\in \CI(Y;\Hom(T^{0,1}_bY,T^{1,0}_bY))
:\:\|\Phi\|_{L^{\infty}}<1\},
\end{equation}
via the prescription:
\begin{equation}
  {}^{\Phi}T^{0,1}_{b,y}Y=\{\bZ_y+\Phi_y(\bZ_y):\: \bZ_y\in T^{0,1}_{b,y}Y\}.
\end{equation}
Here and in the sequel we often use the Szeg\H o projector to label a CR-structure.
Let $\cE\subset\Def(Y,H,\cS_0)$ consist of the embeddable deformations, that
is, CR-structures arising as pseudoconvex boundaries of complex surfaces.
In~\cite{Epstein} we showed that if $\cS_0$ is Szeg\H o projector defined by
the reference CR-structure and $\cS_1$ that defined by an embeddable
deformation, then the map
\begin{equation}
  \cS_1:\Im \cS_0\longrightarrow \Im \cS_1
\end{equation}
is a Fredholm operator. $\Rind(\cS_0,\cS_1)$ denotes its
Fredholm index, which we call the \emph{relative index}.  In the proof
of Theorem E in~\cite{Epstein} we showed that, for each
$m\in\bbN\cup\{0\}$ and any $\delta>0,$ the subsets of
$\Def(Y,H,\cS_0)$ given by
\begin{equation}
  \fS^{\delta}_m=\{\cS_1\in\Def(Y,H,\cS_0):\: -\infty<\Rind(\cS_0,\cS_1)\leq m\}\text{
      and } \|\Phi\|^2_{L^{\infty}}\leq \frac 12-\delta,
\end{equation}
are closed in the $\CI$-topology.  In fact, we show that there is an
integer $k_0,$ so that this conclusion holds for a sequence $<\Phi_n>$
converging to $\Phi$ in the $\cC^{k_0}$-norm. 

Combining~\eqref{UprBnd} with Theorem E of~\cite{Epstein} we prove:
\begin{corollary} Under the hypotheses of Theorem~\ref{thm1}, the
  set of embeddable deformations of the CR-structure on $Y$ is
  closed in the $\CI$-topology.
\end{corollary}

\begin{proof}[Proof of the Corollary] Suppose that $<\Phi_n>$ is a
  sequence of embeddable deformations in $\cE\subset\Def(Y,H,\cS_0)$
  converging to $\Phi\in \Def(Y,H,\cS_0),$  in the $\CI$-topology. We
  first  observe that $\|\Phi\|_{L^{\infty}}<1.$ 

Let $\Psi_1$ and $\Psi_2$ be deformations of the reference
  structure, with local representations
  \begin{equation}
    \Psi_j=\psi_jZ\otimes\bar{\omega}.
  \end{equation}
The local representation of $\Psi_2$ as a deformation of $\Psi_1$ is given by
\begin{equation}\label{eqn12}
  \psi_{21}=\frac{\psi_2-\psi_1}{1-\overline{\psi_1}\psi_2},
\end{equation}
see equation (5.5) in~\cite{Epstein}[I].  We can represent $\Phi$ as a
deformation of any of the structures in the sequence. From
equation~\eqref{eqn12} it is clear that there an integer $N$ so that,
as deformations of $\Phi_N,$ a tail of the sequence and its limit lie
in the $L^{\infty}$-ball in $\Def(Y,H,\cS_N),$ centered at $0,$ of
radius $\frac 14.$ Theorem~\ref{thm1} shows that there is an $M$ so
that
  \begin{equation}
    \Rind(\cS_N,\cS_n)\leq M,\text{ for all }n\in\bbN.
  \end{equation}
Theorem E from~\cite{Epstein} then implies that the limiting structure $\Phi$
is also embeddable, completing the proof of the corollary.
\end{proof}

Before proving Theorem~\ref{thm1} we recall the formula for the relative index proved in~\cite{Epstein44}:
\begin{theorem}\label{thm13} Let $(Y,H)$ be a compact 3-dimensional co-oriented, contact
  manifold, and let $\cS_0,\cS_1$ be Szeg\H o projectors for embeddable
  CR-structures with underlying plane field $H.$ Suppose that $(X_0,J_0),$
  $(X_1,J_1)$ are strictly pseudoconvex complex manifolds with boundaries
  $(Y,H,\cS_0),$ $(Y,H,\cS_1),$ respectively, then
\begin{equation}
\begin{split}
\Rind(\cS_0,\cS_1)=&\dim H^{0,1}(X_0,J_0)-\dim
H^{0,1}(X_1,J_1)+\\
&\frac{\sign[X_0]-\sign[X_1]+\chi[X_0]-\chi[X_1]}{4}.
\end{split}
\label{eqn2}
\end{equation}
\end{theorem}
\noindent
Here $\sign[X]$ is the signature of the non-degenerate quadratic form,
\begin{equation}\label{eqn7}
  ([\alpha],[\beta])\mapsto\int_{X}\alpha\wedge\beta,
\end{equation}
defined for $[\alpha],[\beta]\in \hH^2(X),$ the image of $H^2(X,bX)$ in
$H^2(X),$ and $\chi[X]$ is the topological Euler characteristic:
\begin{equation}
  \chi[X]=\sum_{j=0}^4b_j(X)(-1)^j,\text{ where }b_j(X)=\dim H_j(X;\bbQ).
\end{equation}

\begin{proof}[Proof of Theorem~\ref{thm1}] 
  Let $X_1$ be a minimal resolution of the normal Stein space with
  boundary $(Y,H,\cS_1).$ It follows from a theorem of Bogomolov and
  De Oliveira that there is a small perturbation of the complex
  structure on $X_1$ making it into a Stein manifold,
  see~\cite{BogomolovDeOliveira}. Hence it follows that $X_1,$ with a
  deformed complex structure, has a strictly plurisubharmonic
  exhaustion function, and therefore $X_1$ has the homotopy type of a
  2-dimensional CW-complex. Thus expanding the formula in~\eqref{eqn2}
 gives:
\begin{equation}
\Rind(\cS_0,\cS_1)=C_0-\dim
H^{0,1}(X_1,J_1)-\frac{\sign[X_1]+1-b_1(X_1)+b_2(X_1)}{4},
\label{eqn3}
\end{equation}
where $C_0$ denotes the contribution of the terms from the reference structure:
\begin{equation}
  C_0=H^{0,1}(X_0,J_0)+\frac{\sign[X_0]+\chi(X_0)}{4}.
\end{equation}
The fact that $X_1$ is homotopic to a 2-complex implies that
$b_1(X_1)\leq b_1(Y),$ see~\cite{stipsicz}. As $\sign[X_1]$ is the
signature of the cup product pairing on $\hH^2(X_1),$ it is evident
that
\begin{equation}
  |\sign[X_1]|\leq \dim \hH^2(X_1)\leq\dim H^2(X_1,bX_1)= b_2(X_1).
\end{equation}
The last equality is a consequence of the Lefschetz duality theorem.
Hence 
$0\leq b_2(X_1)+\sign[X_1],$
and therefore
\begin{equation}
  \Rind(\cS_0,\cS_1)\leq C_0+\frac{b_1(Y)-1}{4}.
\end{equation}
This completes the proof of the theorem.
\end{proof}

\noindent
{\bf Remarks on the Ozbagci-Stipsicz Conjecture:}\,
Note that 
$$\sign[X_1]+b_2(X_2)=2b^+_2(X_1)+b^0_2(X_1),$$
where $b^+_2(X_1)$ is the dimension of the space on which the pairing
in~\eqref{eqn7} is positive and $b^0_2(X_1)$ is the dimension of the
kernel of the map $H^2(X_1,bX_1)\to H^2(X_1).$  A global bound
on $|\Rind(\cS_0,\cS_1)|,$ among all Szeg\H o projectors $\cS_1$
defined by elements of $\cE,$ is therefore equivalent to an upper
bound for $b^+_2(X_1)+b^0_2(X_1)+\dim H^{0,1}(X_1),$ among all Stein
spaces, $X_1$ filling $(Y,H).$ The existence of an upper bound on
$b^+_2(X_1)+b^0_2(X_1)$ was conjectured by Ozbagci and Stipsicz, and
proved in some special cases, see~\cite{stipsicz}.

The fact, proved in~\cite{Epstein}, that $\Rind(\cS_0,\cS_1)\geq 0,$
for sufficiently small deformations shows that, for such deformations:
\begin{multline}\label{eqn16.1}
  \dim H^{0,1}(X_1)+\frac{2b_2^+(X_1)+b_2^0(X_1)}{4}\leq\\
  \dim H^{0,1}(X_0)+\frac{2b_2^+(X_0)+b_2^0(X_0)+b_1(Y)-b_1(X_0)}{4}.
\end{multline}
In~\cite{stipsicz} Stipsicz shows that for any Stein filling of $(Y,H),$
we have the estimate $b_2^0(X_1)\leq b_1(Y),$ as well as the existence of a
constant $K_{(Y,H)}$ so that
\begin{equation}
  b_2^{-}(X_1)\leq 5b_2^+(X_1)+2-K_{(Y,H)}+2b_1(Y).
\end{equation}
These estimates, along with~\eqref{eqn16.1} prove a ``germ'' form of the
Ozbagci--Stipsicz conjecture: among sufficiently small, embeddable
deformations of the CR-structure on the boundary of a strictly
pseudoconvex surface, the set of numbers
$$\{b_1(X_1),\sigma(X_1),\chi(X_1)\}$$ 
is finite. The notion of smallness here depends in a complicated way
on the reference CR-structure.

Our results suggest a strategy for proving a lower bound on
$\Rind(\cS_0,\cS_1),$ among deformations $\Phi$ with
$\|\Phi\|_{L^{\infty}}<1-\epsilon,$ for an $\epsilon>0.$ Suppose that
no such bound exists, one could then choose a sequence
$<\Phi_n>\,\subset\cE$ for which $\Rind(\cS_0,\cS_n)$ tends to
$-\infty.$ A contradiction would follow immediately if we could show
that $<\Phi_n>$ is bounded in the $\cC^{k_0+1}$-norm.

While such an \emph{a priori} bound seems unlikely for the original
sequence, it would suffice to replace the sequence $<\Phi_n>$ with a
``wiggle-equivalent'' sequence. Let $M_n$ denote a projective surface
containing $(Y,{}^{\Phi_n}T^{0,1}_bY)$ as a separating hypersurface,
see~\cite{Lempert}. An equivalent sequence with better regularity
might be obtained by wiggling the hypersurfaces defined by
$(Y,{}^{\Phi_n}T^{0,1}_bY)$ within $M_n,$ perhaps using some sort of
heat-flow. After composing the resultant deformations with contact
transformations, we might be able to obtain a sequence $<\Phi'_n>$
with $\Rind(\cS_0,\cS_n')=\Rind(\cS_0,\cS_n)$ that does satisfy an
\emph{a priori} $\cC^{k_0+1}$-bound. Such an argument would seem to
require an improved understanding of the metric geometry of
$\Def(Y,H,\cS_0),$ as well as the relationship of an abstract
deformation to the local extrinsic geometry of $Y$ as a hypersurface
in $M_n.$

{\small \centerline{Acknowledgment} I would like to thank Sylvain
  Cappell for helping me with some topological calculations, and the
  Courant Institute for their hospitality during the completion of
  this work.}

\end{document}